\documentclass[reqno]{amsart}

\usepackage{graphicx}
\usepackage{amscd}
\usepackage{amsmath, amssymb}
\usepackage{graphics}
\usepackage{graphicx}
\usepackage{epsfig}
\usepackage{xcolor}
\usepackage{soul}
\usepackage{tikz}
\usepackage{float}
\usepackage{enumitem}

\newtheorem{theorem}{Theorem}

\newtheorem{definition}[theorem]{Definition}

\newtheorem{lemma}[theorem]{Lemma}

\newtheorem{proposition}[theorem]{Proposition}

\newtheorem{question}[theorem]{Question}

\begin{document}

\title[Classification and Irreducibility of 2-Local Representations of $T_n$]{On the classification and irreducibility of $2$-Local Representations of the twin group $T_n$}

\author{Taher I. Mayassi}

\address{Taher I. Mayassi\\
Department of Mathematics and Physics\\
Lebanese International University\\
P.O. Box 146404, Beirut, Lebanon}

\email{taher.mayasi@liu.edu.lb}

\author{Mohamad N. Nasser}

\address{Mohamad N. Nasser\\
         Department of Mathematics and Computer Science\\
         Beirut Arab University\\
         P.O. Box 11-5020, Beirut, Lebanon}
         
\email{m.nasser@bau.edu.lb}

\begin{abstract}
We investigate the homogeneous $2$-local representations of the twin group $T_n$ for all integers $n\geqslant 2$. A complete classification is obtained, yielding three distinct families of representations. We show that each of these families is reducible by explicitly constructing one-dimensional invariant subspaces, with particular emphasis on the first family, namely $\xi_1: T_n \rightarrow \text{GL}_n(\mathbb{C})$. Passing to the corresponding quotients, we construct a reduced representation of $\xi_1$, namely $\tilde{\xi}_1: T_n \rightarrow \text{GL}_{n-1}(\mathbb{C})$. The core of the paper is that we establish, through a precise criterion, a necessary and sufficient condition for the irreducibility of the representation $\tilde{\xi}_1$.
\end{abstract}

\maketitle

\renewcommand{\thefootnote}{}
\footnote{\textit{Key words and phrases.} Twin group, Braid group, Local representations, Irreducibility.}
\footnote{\textit{Mathematics Subject Classification.} Primary: 20F36.}

\vspace*{-0.35cm}

\section{Introduction}

\vspace*{0.1cm}
For a positive integer $n$, the twin group $T_n$ on $n$ strands admits a topological interpretation that is similar to that of the braid groups. Consider the strip $\mathbb{R}\times[0,1]$ and fix $n$ marked points on each of the horizontal lines $y=0$ and $y=1$. A twin is represented by $n$ monotonic (descending) strands within this strip, connecting the $n$ top points to the $n$ bottom points, such that no three strands meet at the same point \cite{Khova,Nai}. The following figure is an example of a twin of $5$ strands.
\begin{center}
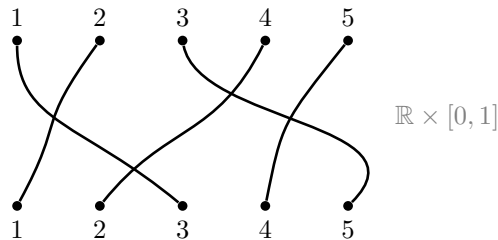
\begin{figure}[h]
\begin{tikzpicture}[scale=1.1,
  endpoint/.style={circle,fill=black,inner sep=1.3pt}]
  \foreach \i in {0,...,4}{
   \pgfmathtruncatemacro{\j}{\i+1}
    \node[endpoint,label=above:{\j}] (T\i) at (\i,2) {};
    \node[endpoint,label=below:{\j}] (B\i) at (\i,0) {};
  }
  \draw[line width=1pt] (T0) .. controls +(0, -1) and +(-1.2, 1) .. (B2);
  \draw[line width=1pt] (T1) .. controls +(-0.8, -1.1) and +(0.6,  1.1) .. (B0);
  \draw[line width=1pt] (T2) .. controls +(0.2, -1.0) and +( 1.0,  1.0) .. (B4);
  \draw[line width=1pt] (T3) .. controls +(-0.6,-1.2) and +( 0.8,  1.0) .. (B1);
  \draw[line width=1pt] (T4) .. controls +(-0.8,-1.0) and +( 0.2,  1.0) .. (B3);
  \node[gray] at (5.2,1.07) {$\mathbb{R}\times[0,1]$};
\end{tikzpicture}
\caption{$5$-strand twin}
\end{figure}
\end{center}
\vspace*{-0.3cm}
\noindent Two such configurations are equivalent if they are related by a homotopy (fixing the endpoints) that preserves this condition \cite{Khova,Nai}. The set of equivalence classes of these configurations forms a group under concatenation (stacking one configuration atop another and rescaling vertically). This group is the twin group $T_n$.

\vspace{0.2cm}

The twin groups have recently attracted considerable attention, both for their rich algebraic structure and for their connections to problems in low-dimensional topology and combinatorics. From a geometric perspective, elements of $T_n$ can be realized by doodles, that is, collections of immersed closed curves in the plane that do not admit triple intersections. Such doodle diagrams provide a natural combinatorial and topological model for the group, offering an alternative way to study its properties beyond the purely algebraic description \cite{Lopa, Khova}. In this setting, doodles are similar to braid diagrams, but the analogy is not exact: the allowed moves are different, which leads to different isotopy classes and a theory that departs from the classical braid groups.

\vspace{0.2cm}

While braid groups are classical objects with deep connections to knot theory, mapping class groups, and quantum topology \cite{Bir}, twin groups remain comparatively less explored. Algebraically, both groups share a presentation with generators $s_1,\dots,s_{n-1}$ and braid-like relations, yet in $T_n$ the braid relation $s_is_{i+1}s_i=s_{i+1}s_is_{i+1}$ is replaced by the involutivity condition $s_i^2=1$, leading to substantially different algebraic and representation-theoretic properties \cite{Lopa}.

\vspace{0.2cm}

Despite their recent introduction, some linear representations of $T_n$ have been constructed and studied \cite{Lopa_2, M.N.twin}, motivated by potential applications in topology, combinatorics, and the study of motion groups. In representation theory, two fundamental properties often investigated are faithfulness — whether the representation is injective — and irreducibility — whether the representation admits no proper non-trivial invariant subspaces. Both properties play central roles in understanding the algebraic structure of a group and in applications to topology, geometry, and mathematical physics \cite{Serr,Hump}.

\vspace{0.2cm}

A particularly interesting class of representations is given by local representations, in which each generator acts non-trivially on a small block of coordinates while fixing the rest. Such representations have been studied for braid groups, Coxeter groups, and related motion groups \cite{Funa, Cris}, with applications to link invariants, statistical mechanics, and homological representations. Classification results for local representations often produce families that are then analyzed for faithfulness, reducibility, and other structural features.

\vspace{0.2cm}

In this paper, we focus on the 2-local representations of $T_n$ for $n\geqslant 3$. In section 2 of the paper, we introduce generalities and basic results we need in our work. In section 3, we classify all such representations, obtaining three distinct families. In section 4, we show that all these representations are reducible by explicitly constructing a one-dimensional invariant subspace; in particular, we identify such subspaces for the first family $\xi_1: T_n \rightarrow \text{GL}_n(\mathbb{C})$. In addition, we focus on the reduction of the first family $\xi_1$, namely $\tilde{\xi}_1: T_n \rightarrow \text{GL}_{n-1}(\mathbb{C})$, by factoring out the corresponding invariant subspaces. Section 5, which forms the core of the paper, contains three propositions that culminate in a necessary and sufficient condition for the irreducibility of $\tilde{\xi}_1$. The main result is that we prove that the representation $\tilde{\xi}_1$ of $T_n$, $n\geqslant4$, is irreducible if and only if $a\not\in\{1,-1\}$ and $a$ is not a root of the polynomial $P(t)=4(1+t^2) + \frac{(1-t)^4}{2t} \left( 1 - \left( \frac{1-t}{1+t} \right)^{n-4} \right).$

\vspace{0.3cm}

\section{Generalities}

\vspace*{0.1cm}

For any positive integer $n\geqslant 2$, the twin group $T_n$ is presented by $n-1$ generators $s_1, \dots, s_{n-1}$ and the following relations.
$$\begin{array}{ccccl}
s_i^2 &=&1 &\text{ for } & 1\leqslant i\leqslant n-1,\\
s_is_j &=& s_js_i &\text{ for } & |i-j|>1.
\end{array}
$$
The generator $s_i$ can be described by the following diagram.
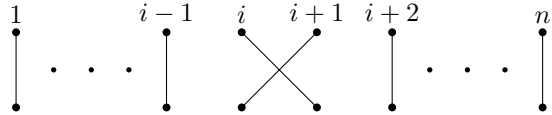
\begin{figure}[H]
 \begin{tikzpicture}
 \draw (0,0)--(0,1)(2,0)--(2,1)(5,0)--(5,1)(3,0)--(4,1)(4,0)--(3,1)(7,1)--(7,0);
 \foreach \n in {0,2,3,4,5,7} \fill (\n,1)circle(1.5pt);
 \foreach \n in {0,2,3,4,5,7} \fill (\n,0)circle(1.5pt);
 \foreach \n in {.5,1,1.5,5.5,6,6.5} \fill (\n,0.5)circle(1pt);
 \node[above] at (0,1) {$1$};
 \node[above] at (2,1) {$i-1$};
 \node[above] at (3,1) {$i$};
  \node[above] at (4,1) {$i+1$};
  \node[above] at (5,1) {$i+2$};
  \node[above] at (7,1) {$n$};
  \end{tikzpicture}
  \caption{The generator $s_i$} 
  \end{figure} 
\vspace*{-0.3cm}
\noindent A closure of a twin $t$ is a diagram obtained by joining the end points of $t$ as shown in the figure below \cite{Nai}.
\vspace*{-0.2cm}
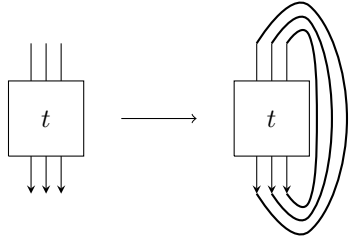
\begin{figure}[H]
\begin{tikzpicture}
\draw (0,0)--(1,0)--(1,1)--(0,1)--(0,0);
\node at (.5,.5){$t$};
\draw (0.3,1.5)--(0.3,1)(0.5,1.5)--(0.5,1)(0.7,1.5)--(0.7,1);
\draw[-stealth] (0.3,0)--(0.3,-.5);
\draw[-stealth](0.5,0)--(0.5,-0.5);
\draw[-stealth](0.7,0)--(0.7,-.5);
\draw[->] (1.5,.5)--(2.5,.5);
\draw (3,0)--(4,0)--(4,1)--(3,1)--(3,0);
\node at(3.5,.5) {$t$};
\draw (3.3,1.5)--(3.3,1)(3.5,1.5)--(3.5,1)(3.7,1.5)--(3.7,1);
\draw[-stealth] (3.3,0)--(3.3,-.5);
\draw[-stealth](3.5,0)--(3.5,-0.5);
\draw[-stealth](3.7,0)--(3.7,-.5);
\draw[thick]
  plot [smooth, tension=0.8] coordinates {
    (3.3,1.5)
    (4,2)
    (4.5,0.5)
    (4,-1)
    (3.3,-0.5)
  };
  \draw[thick]
  plot [smooth, tension=0.8] coordinates {
    (3.5,1.5)
    (4,1.8)
    (4.3,0.5)
    (4,-0.8)
    (3.5,-0.5)
  };
 \draw[thick]
  plot [smooth, tension=0.8] coordinates {
    (3.7,1.5)
    (4,1.6)
    (4.1,0.5)
    (4,-0.6)
    (3.7,-0.5)
  };
\end{tikzpicture}
\caption{The closure of the twin $t$} 
\end{figure}
\vspace*{-0.4cm}
\begin{definition}\cite{Nai}
A doodle is a collection of piecewise-linear closed curves $\{C_1,\dots,  C_n\}$ on a closed oriented surface such that there is no common point of any three curves and no triple intersection of the same curve $C_i$.
\end{definition}

An oriented doodle is a doodle whose curves carry orientations. Two doodles on a closed surface are equivalent if related by a homotopy without triple intersections, which is generated by moves $M_1$ and $M_2$ \cite{Nai}.

\vspace*{0.1cm}

\begin{center}
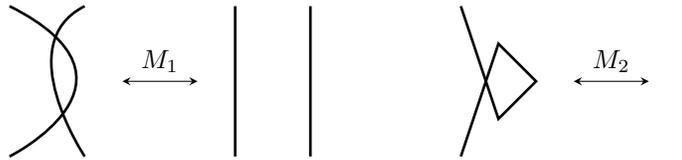
\begin{figure}[H]
\begin{tikzpicture}[scale=1,
  endpoint/.style={circle,fill=black,inner sep=1.3pt}
]
\draw[line width=1pt] (0,2) .. controls +(2,-1)and +(0,0) .. (0,0);
\draw[line width=1pt] (1,2) .. controls +(-1,-.5)and +(0,0) .. (1,0);
\draw[stealth-stealth] (1.5,1)--(2.5,1);
\node[above] at (2,1){$M_1$};
\draw[line width=1pt] (3,0)--(3,2);
\draw[line width=1pt] (4,0)--(4,2);
\draw[line width=1pt] (6,2)--(6.5,.5)--(7,1)--(6.5,1.5)--(6,0);
\draw[stealth-stealth] (7.5,1)--(8.5,1);
\draw[line width=1pt] (9,2)--(9,0);
\node[above] at (8,1){$M_2$};
\end{tikzpicture}
\caption{The $M_1$ and $M_2$ moves}
\end{figure}
\end{center}

\vspace*{-0.5cm}

It is clear that the closure of a twin is a doodle. The following theorem shows the converse.
\begin{theorem} \cite{Nai}
Every oriented doodle is a closure of a twin.
\end{theorem}

\vspace*{0.3cm}

In what follows, we give the concept of $k$-local representations of any group $G$ with finite number of generators. 
\vspace*{0.1cm}

\begin{definition}\cite{Nas20241}
Let $G$ be a group with generators $g_1,g_2,\ldots,g_{n-1}$. A representation $\theta: G \rightarrow \text{GL}_{m}(\mathbb{C})$ is said to be $k$-local if it is of the form
$$\theta(g_i) =\left( \begin{array}{c|@{}c|c@{}}
   \begin{matrix}
     I_{i-1} 
   \end{matrix} 
      & 0 & 0 \\
      \hline
    0 &\hspace{0.2cm} \begin{matrix}
   		M_i
   		\end{matrix}  & 0  \\
\hline
0 & 0 & I_{n-i-1}
\end{array} \right) \hspace*{0.2cm} \text{for} \hspace*{0.2cm} 1\leq i\leq n-1,$$ 
where $M_i \in \text{GL}_k(\mathbb{C})$ with $k=m-n+2$ and $I_r$ is the $r\times r$ identity matrix. The $k$-local representation is said to be homogeneous if all the matrices $M_i$ are equal.
\end{definition}

\vspace*{0.1cm}

Y. Mikhalchishina classified all $2$-local linear representations of $B_3$ and further developed all homogeneous $2$-local representations of $B_n$, for $n \geqslant 3$ \cite{Mik2013}. Moreover, T. Mayassi and M. Nasser generalized Mikhalchishina's work by classifying all homogeneous $3$-local representations of $B_n$ for $n\geqslant 4$ \cite{Mayassi2025}. Additionally, M. Nasser, M. Chreif, and M. Dally determined all homogeneous $k$-local representations of the flat virtual braid group $FVB_n$ for different degrees $k$ \cite{MNas2025}. Also, M. Nasser, V. Keshari, and M. Prabhakar classified all homogeneous $2$-local representations of the twisted virtual braid group $TVB_n, n\geqslant 2$, studying their faithfulness and irreducibility in some cases \cite{Mnas2025}. Building on this, V. Keshari, M. Nasser, and M. Prabhakar investigated all homogeneous $2$-local representations of the multi-virtual braid group $M_kVB_n$ and the multi-welded braid group $M_kWB_n$ \cite{nassernew}. All these previous works were to answer the following question.

\vspace*{0.1cm}

\begin{question}
If $\rho$ is a homogeneous $k$-local representation of a group $G$, then what are the possible forms and properties of $\rho$?
\end{question}

\vspace*{0.1cm}

In this paper, we answer this question for the twin group $T_n$ for all $n \geqslant 3$ in the case $k=2$.

\vspace*{0.2cm}

\section{Classification of all homogeneous $2$-local representations of the twin groups $T_n$}

\vspace{.1cm}

The following theorem provides a complete classification of all homogeneous $2$-local representations of the twin group $T_n$ for every $n\geqslant 2$.

\vspace{.1cm}

\begin{theorem} \label{TheoTn}
For $n\geqslant 2$, let $\xi: T_n \rightarrow \text{GL}_n(\mathbb{C})$ be a non-trivial homogeneous $2$-local representation of $T_n$. Then, $\xi$ is equivalent to one of the following three representations.

\vspace{.2cm}

\begin{itemize}
\item[(1)] $\xi_1: T_n \rightarrow \text{GL}_n(\mathbb{C})$ such that
$$\xi_1(s_1) =\left(\begin{array}{c|@{}|c@{}}
   \begin{matrix}
     a & b\\
   \frac{1-a^2}{b} & -a 
   \end{matrix} 
     & 0 \\
      \hline
0 & I_{n-2}
\end{array}
\right),\hspace{0.5cm}
\xi_1(s_{n-1}) =\left( \begin{array}{c|@{}|c@{}}
   I_{n-2} & 0 \\
      \hline
0 & \begin{matrix}
     a & b\\
   \frac{1-a^2}{b} & -a 
   \end{matrix} 
\end{array}
\right),$$
and for $1<k<n-1$,
$$\xi_1(s_k) =\left( \begin{array}{c|@{}c|c@{}}
   \begin{matrix}
     I_{k-1} 
   \end{matrix} 
      & 0 & 0 \\
      \hline
    0 &\hspace{0.2cm}  \begin{matrix}
   		a & b\\
   		\frac{1-a^2}{b} & -a
   		\end{matrix}  & 0  \\
\hline
0 & 0 & I_{n-k-1}
\end{array} \right),$$
where $a,b\in \mathbb{C}$, $b\neq 0$, and $I_i$ is the identity matrix of order $i$.\\

\item[(2)] $\xi_2: T_n \rightarrow \text{GL}_n(\mathbb{C})$ such that
$$\xi_2(s_1) =\left(\begin{array}{c|@{}|c@{}}
   \begin{matrix}
     \pm1 & 0\\
   c & \mp1 
   \end{matrix} 
     & 0 \\
      \hline
0 & I_{n-2}
\end{array}
\right),\hspace{0.5cm}
\xi_2(s_{n-1}) =\left( \begin{array}{c|@{}|c@{}}
   I_{n-2} & 0 \\
      \hline
0 & \begin{matrix}
     \pm1 & 0\\
   c & \mp1 
   \end{matrix}
\end{array}
\right),$$
and for $1<k<n-1$,
$$\xi_2(s_k) =\left( \begin{array}{c|@{}c|c@{}}
   \begin{matrix}
     I_{k-1} 
   \end{matrix} 
      & 0 & 0 \\
      \hline
    0 &\hspace{0.2cm}  \begin{matrix}
   		\pm1 & 0\\
   		c & \mp1
   		\end{matrix}  & 0  \\
\hline
0 & 0 & I_{n-k-1}
\end{array} \right),$$ where $c\in \mathbb{C}$.
\\
\item[(3)] $\xi_3: T_n \rightarrow \text{GL}_n(\mathbb{C})$ such that
$$\xi_3(s_1) =\left(\begin{array}{c|@{}|c@{}}
   \begin{matrix}
     -1 & 0\\
   0 & -1 
   \end{matrix} 
     & 0 \\
      \hline
0 & I_{n-2}
\end{array}
\right),\hspace{0.5cm}
\xi_3(s_{n-1}) =\left( \begin{array}{c|@{}|c@{}}
   I_{n-2} & 0 \\
      \hline
0 & \begin{matrix}
     -1 & 0\\
   0 & -1 
   \end{matrix} 
\end{array}
\right),$$
and for $1<k<n-1$,
$$\xi_3(s_k) =\left( \begin{array}{c|@{}c|c@{}}
   \begin{matrix}
     I_{k-1} 
   \end{matrix} 
      & 0 & 0 \\
      \hline
    0 &\hspace{0.2cm}  \begin{matrix}
   		-1 & 0\\
   		0 & -1
   		\end{matrix}  & 0  \\
\hline
0 & 0 & I_{n-k-1}
\end{array} \right).$$
\end{itemize}
\end{theorem}

\begin{proof}
Because $\xi$ is a $2$-local representation of $T_n$, it follows that there is a $2\times2$ invertible complex matrix $M=\left( \begin{array}{@{}c@{}}
  \begin{matrix}
   		a & b\\
   		c & d
   		\end{matrix}
\end{array} \right)$ such that
$$\xi_1(s_1) =\left(\begin{array}{c|@{}|c@{}}
M & 0 \\
\hline
0 & I_{n-2}
\end{array}
\right),\hspace{0.5cm}
\xi_1(s_{n-1}) =\left( \begin{array}{c|@{}|c@{}}
I_{n-2} & 0 \\
\hline
0 & M 
\end{array}
\right),$$
and for $1<k<n-1$,
$$\xi_1(s_k) =\left( \begin{array}{c|c|c@{}}
   \begin{matrix}
     I_{k-1} 
   \end{matrix} 
      & 0 & 0 \\
      \hline
    0 & M & 0  \\
\hline
0 & 0 & I_{n-k-1}
\end{array} \right).$$
Note that for $i-j>1$, it is clear that
$$\xi(s_i)\xi(s_j)=\xi(s_j)\xi(s_i)=I_{j-1}\oplus M\oplus I_{i-j-2}\oplus M \oplus I_{n-i-1}.$$
So, the following relations hold for all $a,b,c,d\in \mathbb{C}.$
$$\xi(s_i)\xi(s_j)=\xi(s_j)\xi(s_i) \text{ whenever }|i-j|>1.$$
Since for $1\leqslant k\leqslant n-1$, $s_k^2=1$, the identity of $T_n$, it follows that $\xi(s_k)^2=I_n$. 
$$\xi(s_1)^2=\left(\begin{array}{c|@{}|c@{}}
   \begin{matrix}
     a^2+bc & ab+bd\\
   ac+cd & d^2+bc 
   \end{matrix} 
     & 0 \\
      \hline
0 & I_{n-2}
\end{array}
\right), \hspace{0.5cm}
\xi(s_{n-1})^2=\left(\begin{array}{c|@{}|c@{}}
    I_{n-2}
     & 0 \\
      \hline
0 & \begin{matrix}
     a^2+bc & ab+bd\\
   ac+cd & d^2+bc 
   \end{matrix}
\end{array}
\right),
$$
and 
$$\xi(s_k)^2=\left( \begin{array}{c|@{}c|c@{}}
   \begin{matrix}
     I_{k-1} 
   \end{matrix} 
      & 0 & 0 \\
      \hline
    0 &\hspace{0.2cm} \begin{matrix}
     a^2+bc & ab+bd\\
   ac+cd & d^2+bc 
   \end{matrix} & 0  \\
\hline
0 & 0 & I_{n-k-1}
\end{array} \right)\;\;\; \text{ for all } 1<k<n-1.$$
Thus the relations: $\xi(s_k)^2=I_n$ for all $1\leqslant k\leqslant n-1$ yield the following equations:
\begin{equation}
a^2+bc=1
\end{equation}
\begin{equation}
(a+d)b=0
\end{equation}
\begin{equation}
(a+d)c=0
\end{equation}
\begin{equation}
d^2+bc=1
\end{equation}
Elementary computations of this system of equations imply the required result.
\end{proof}

\vspace{0.3cm}

\section{Analysis of the homogeneous $2$-local representations of $T_n$}

In this section we study the structural properties of the homogeneous $2$-local representations of the twin group $T_n$. We first investigate their reducibility and establish that all such representations admit a non-trivial invariant subspace. In the second part, we explicitly perform the reduction process by quotienting out these invariant subspaces, thereby obtaining reduced representations of dimension $n-1$ and describing their matrix forms.

\vspace{.1cm}

\subsection{Reducibility of the homogeneous $2$-local representations of $T_n$}

\vspace{0.1cm}

Here, we show that all homogeneous $2$-local representations of $T_n$ are reducible. We omit the discussion of the reducibility of the representations $\xi_2$ and $\xi_3$ given in Theorem \ref{TheoTn}, as these cases are straightforward. For the rest of the paper, the vectors in $\mathbb{C}^n$ are considered as column vectors.

\begin{proposition}\label{xi1}
Let $\xi_1:T_n\to \text{GL}_n(\mathbb{C})$ be the representation of the twin group $T_n$ given in Theorem \ref{TheoTn}. Then $\xi_1$ is reducible for all $a,b\in\mathbb{C}$ with $b\neq0$.
\end{proposition}

\begin{proof}
Consider the subspace $S$ of $\mathbb{C}^n$ spanned by the vector $$v=e_1+\frac{1-a}{b}e_2+
\frac{(1-a)^2}{b^2}e_3+\dots+
\frac{(1-a)^{n-1}}{b^{n-1}}e_n,$$
where $\{e_1,e_2,\dots, e_n\}$ is the canonical basis for $\mathbb{C}^n$.
Note that for $1\leqslant k\leqslant n-1$ we have 
$$
\left(
\begin{array}{cc}
 a & b \\
 \frac{1-a^2}{b} & -a \\
\end{array}
\right)
\left(
\begin{array}{c}
 \frac{(1-a)^{k-1}}{b^{k-1}} \\
 \frac{(1-a)^{k}}{b^{k}} \\
\end{array}
\right)
=\left(
\begin{array}{c}
 \frac{(1-a)^{k-1}}{b^{k-1}} \\
 \frac{(1-a)^{k}}{b^{k}} \\
\end{array}
\right).
$$ 
Therefore, 
$$\xi_1(s_k)v=(I_{k-1}\oplus\left(
\begin{array}{cc}
 a & b \\
 \frac{1-a^2}{b} & -a \\
\end{array}
\right)\oplus I_{n-k-1})v=v \text{ for all } 1\leqslant k\leqslant n-1.$$
This implies that $S$ is invariant under $\xi_1(s_k)$ for all $1\leqslant k \leqslant n-1$ and for all $a,b\in\mathbb{C}$ such that $b\neq0$. Hence, $\xi_1$ is reducible.
\end{proof}

\subsection{Reduction of the homogeneous 2-local representations of $T_n$}

\vspace{0.1cm}
We reduce the representation $\xi_1$ given in Theorem \ref{TheoTn}. We have shown that the vector
 $$v=e_1+\frac{1-a}{b}e_2+\frac{(1-a)^2}{b^2}e_3+\dots+\frac{(1-a)^{n-1}}{b^{n-1}}e_n$$
spans a one-dimensional subspace of $\mathbb{C}^n$ that is invariant under the action of the twin group $T_n$ in the representation $\xi_1$ for all $a,b\in\mathbb{C}$ with $b\neq 0$. By quotienting out this invariant subspace, we obtain an $(n-1)$-dimensional representation of the twin group $T_n$. To write the matrices of the reduced representation explicitly, we change from the canonical basis $\{e_1,\dots, e_n\}$ of $\mathbb{C}^n$ to the basis $\{v,e_2,\dots,e_n\}$. Let $Q$ denote the corresponding change-of-basis matrix:
$$Q=(v,e_2,e_3,\dots,e_n)=\left(
\begin{array}{ccccc}
 1 & 0 & 0 & \cdots & 0  \\
 \frac{1-a}{b} & 1 & 0 & \cdots & 0  \\
 \frac{(1-a)^2}{b^2} & 0 & 1 & \cdots & 0  \\
 \vdots & \vdots & \vdots &  & \vdots  \\
 \frac{(1-a)^{n-1}}{b^{n-1}} & 0 & 0 & \cdots & 1 \\
\end{array}
\right)=I_n+(v-e_1)e_1^T,$$
where $T$ stands for transpose of a matrix.
Then the inverse of $Q$ is given as
$$Q^{-1}=I_n-(v-e_1)e_1^T=\left(
\begin{array}{ccccc}
 1 & 0 & 0 & \cdots & 0  \\
- \frac{1-a}{b} & 1 & 0 & \cdots & 0  \\
 -\frac{(1-a)^2}{b^2} & 0 & 1 & \cdots & 0  \\
 \vdots & \vdots & \vdots &  & \vdots  \\
 -\frac{(1-a)^{n-1}}{b^{n-1}} & 0 & 0 & \cdots & 1 \\
\end{array}
\right).$$
This follows from the Sherman–Morrison formula for rank-one updates of the identity $(I+uv^T)^{-1}=I-uv^T$ if $v^Tu=0$ (see \cite{Golub}).

\vspace*{0.1cm}

Now, the images of the generators $s_k, 1\leqslant k\leqslant n-1$, of $T_n$ under $\xi_1$ relative to the basis $\{v,e_2,e_3,\cdots,e_n\}$ of $\mathbb{C}^n$ are given by:
$$ s_1\mapsto Q^{-1}\xi_1(s_1)Q=
\left(
\begin{array}{ccccccc}
 1 & b & 0 & \cdots& 0 & \cdots & 0  \\
 0 & -1 & 0 & \cdots & 0 & \cdots & 0\\
 0 & \frac{(a-1)^2}{-b} & 1 & \cdots & 0 & \cdots & 0  \\
 \vdots & \vdots & \vdots & \cdots & \vdots &  & \vdots  \\
  0 & \frac{(a-1)^{j}}{(-b)^{j-1}} & 0 & \cdots & 1 & \cdots & 0 \\
  \vdots & \vdots & \vdots & \cdots & \vdots &  & \vdots  \\
  0 & \frac{(a-1)^{n-1}}{(-b)^{n-2}} & 0 & \cdots& 0 & \cdots & 1 \\
\end{array}
\right) 
$$
and
$$ s_k\mapsto Q^{-1}\xi_1(s_k)Q=\xi_1(s_k)\text{ for all } 2\leqslant k\leqslant n-1.
$$

\vspace{0.1cm}

The restriction $\tilde{\xi}_1$ of $\xi_1$ to the quotient space $\mathbb{C}^n/\langle v\rangle\equiv \mathbb{C}^{n-1}$ is obtained by deleting the first row and the first column of each of the matrices $Q^{-1}\xi_1(s_k)Q$, $1\leqslant k\leqslant n-1$, as follows.
$$ s_1\mapsto \left(
\begin{array}{cccccc}
 -1 & 0 & \cdots & 0 & \cdots & 0\\
 \frac{(a-1)^2}{-b} & 1 & \cdots & 0 & \cdots & 0  \\
  \vdots & \vdots & \cdots & \vdots &  & \vdots  \\
  \frac{(a-1)^{j}}{(-b)^{j-1}} & 0 & \cdots & 1 & \cdots & 0 \\
  \vdots & \vdots & \cdots & \vdots &  & \vdots  \\
  \frac{(a-1)^{n-1}}{(-b)^{n-2}} & 0 & \cdots& 0 & \cdots & 1 \\
\end{array}
\right), \;\;\;
s_2\mapsto \left(
\begin{array}{ccccc}
  a & b & 0 & \cdots & 0\\
  \frac{1-a^2}{b} & -a & 0 & \cdots & 0  \\
  0 &  0 & 1 & \cdots & 0\\
 \vdots & \vdots & \vdots & \ddots  & \vdots  \\
  0 &  0 & 0  & \cdots & 1\\
\end{array}
\right),
$$
and
$$
s_k\mapsto \left(
\begin{array}{ccccccccc}
  1 & 0 & 0 & 0 &\cdots & \cdots & \cdots & \cdots & 0\\
  0 & 1 & 0 & 0 &\cdots & \cdots & \cdots &\cdots & 0\\
  0 & 0 & 1 & 0 &\cdots & \cdots & \cdots &\cdots & 0\\
  \vdots & \vdots & \vdots &\ddots & \vdots &\vdots & \vdots &\vdots & \vdots\\
  0 & 0 & 0 & 0 & a & b & 0 & \cdots & 0\\
  0 & 0 & 0 & 0 & \frac{1-a^2}{b} & -a & 0 & \cdots & 0\\
    0 & 0 & 0 & 0 & 0 & 0 & 1 & \cdots & 0\\
  \vdots & \vdots & \vdots &\vdots & \vdots & \vdots & \vdots &\ddots & \vdots\\ 
  0& 0& 0 & 0 & 0 & 0 & 0 & \cdots & 1
  
\end{array}
\right)=
\left( \begin{array}{c|@{}c|c@{}}
   \begin{matrix}
     I_{k-2} 
   \end{matrix} 
      & 0 & 0 \\
      \hline
    0 &\hspace{0.2cm}  \begin{matrix}
   		a & b\\
   		\frac{1-a^2}{b} & -a
   		\end{matrix}  & 0  \\
\hline
0 & 0 & I_{n-k-1}
\end{array} \right).
$$

\vspace*{0.1cm}

Now, we state the following definition.

\vspace*{0.1cm}

\begin{definition}
The representation $\xi_1$ of $T_n$ given in Theorem \ref{TheoTn} is reduced to the representation $\tilde{\xi}_1$ of dimension $n-1$ that is given explicitly above.
\end{definition}

\vspace{0.2cm}

\section{Irreducibility of the representations $\tilde{\xi}_1$ of the twin group}

\vspace{0.2cm}

In this section, we establish the necessary and sufficient condition for the irreducibility of the representation $\tilde{\xi}_1$. To this end, we first present several propositions, lemmas, and theorems, which will lead us to the statement of our main result. We start by the following proposition.

\begin{proposition}\label{prop -1,1}
If $a\in\{1,-1\}$, then the representation $\tilde{\xi}_1$ is reducible.
\end{proposition}

\begin{proof}
We consider two cases in the following.
\begin{enumerate}
\item If $a=1$ then the representation $\tilde{\xi}_1$ of $T_n$ has the subspace spanned by $e_{1}=(1,0,\dots,0)^T$ as an invariant subspace. Indeed, $\tilde{\xi_1}(s_1)e_1=-e_1$ and $\tilde{\xi_1}(s_k)e_1=e_1$ for all $2\leqslant k\leqslant n-1$.
\item If $a=-1$ then the representation $\tilde{\xi}_1$ of $T_n$ has the subspace spanned by $v=\left(\frac{b^{n-2}}{2^{n-2}},\dots, \frac{b}{2},1\right)^T=\sum_{k=1}^{n-1}(\frac{b}{2})^{n-1-k}e_k$ as an invariant subspace. In fact, $\tilde{\xi_1}(s_1)v=-v$ and $\tilde{\xi_1}(s_k)v=v$ for all $2\leqslant k\leqslant n-1$.
\end{enumerate}
\end{proof}

\subsection{Irreducibility of $\tilde{\xi}_1$ of the twin group $T_3$}

In this subsection, we deal with the case $n=3$. The next theorem states a necessary and sufficient condition for the representation $\tilde{\xi}_1$ of $T_3$ to be irreducible.

\begin{theorem}
The representation $\tilde{\xi}_1$ of the twin group $T_3$ is irreducible if and only if $a\not\in\{\pm1,\pm i\sqrt{3}\}$.
\end{theorem}

\begin{proof}
The representation $\tilde{\xi}_1$ of $T_3$ is given by:
$$\tilde{\xi}_1(s_1)=\left(
\begin{array}{cc}
  -1 & 0\\
  -\frac{(a-1)^2}{b} & 1 \\
\end{array}
\right)\;\;\;
\text{ and }\;\;\;
\tilde{\xi}_1(s_2)=\left(
\begin{array}{cc}
  a & b\\
 \frac{1-a^2}{b} & -a \\
\end{array}
\right).$$
The eigenvalues of $\tilde{\xi}_1(s_1)$ are $-1$ and $1$. The vectors
$v_1=\left(\frac{2b}{(a-1)^2}, 1\right)^T$ and $v_2=(0,1)^T$ are two linearly independent eigenvectors of $\tilde{\xi}_1(s_1)$ corresponding to the eigenvalues $-1$ and $1$ respectively.
Consider the matrix
$$P=(v_1,v_2)=\left(
\begin{array}{cc}
 \frac{2b}{(a-1)^2} & 0 \\
  1 & 1\\
\end{array}
\right).$$
Then, the matrices representing $\tilde{\xi}_1(s_1)$ and $\tilde{\xi}_1(s_2)$ relative to the basis $\{v_1,v_2\}$ of $\mathbb{C}^2$ are as follows.
$$\tilde{\xi}_1(s_1)\mapsto S_1=\left(
\begin{array}{cc}
 -1 & 0 \\
  0 & 1\\
\end{array}
\right)$$
and
$$
\tilde{\xi}_1(s_2)\mapsto S_2=P^{-1}\tilde{\xi}_1(s_2)P=
\left(
\begin{array}{cc}
 \frac{1}{2}(1+a^2) & \frac{1}{2}(a-1)^2 \\
  \frac{3+3a+a^2+a^3}{2-2a}&-\frac{1}{2}(1+a^2) \\
\end{array}
\right).$$
The eigenvectors of $\tilde{\xi}_1(s_1)$ relative to the basis $\{v_1,v_2\}$ are written as $v_1=(1,0)^T$ and $v_2=(0,1)^T$. So we have $S_2v_1=\left(
\begin{array}{c}
 \frac{1}{2} \left(a^2+1\right) \\
 \frac{a^3+a^2+3a+3}{2-2 a} \\
\end{array}
\right).$
Then, the subspace $\langle v_1\rangle$ is invariant under the action of $S_2$ if and only if $a=-1$ or $a=\pm i\sqrt{3}$.
In addition, we have also that $S_2v_2=\left(
\begin{array}{c}
 \frac{1}{2} (a-1)^2 \\
 -\frac{1}{2} \left(a^2+1\right) \\
\end{array}
\right)$. Thus, the subspace $\langle v_2\rangle$ is invariant under the action of $S_2$ if and only if $a=1$.

\vspace*{.1cm}

\noindent Therefore, $\tilde{\xi}_1$ is irreducible if and only if $a\not\in\{\pm1,\pm i\sqrt{3}\}$.
\end{proof}

\subsection{Irreducibility of $\tilde{\xi}_1$ of the twin group $T_n$ for $n\geqslant4$}
Now, we deal with the case $n\geqslant 4$. First of all, we note that $\det(\lambda I_{n-1}-\tilde{\xi}_1(s_1))=(\lambda+1)(\lambda-1)^{n-2}$. This implies that the eigenvalues of $\tilde{\xi}_1(s_1)$ are $-1$ (of multiplicity 1) and $1$ (of multiplicity $n-2$). Assume that $a\not\in\{1,-1\}$, then, direct computations yield that the following vectors are $n-1$ linearly independent eigenvectors of $\tilde{\xi}_1(s_1)$ corresponding to the eigenvalues $-1$, $\underbrace{1,\dots, 1}_{n-2}$ respectively.
$$w=\left(
\begin{array}{c}
 \frac{2 b^{n-2}}{(1-a)^{n-1}} \\
 \frac{b^{n-3}}{(1-a)^{n-3}} \\
 \vdots \\
 \frac{b^{n-j-1}}{(1-a)^{n-j-1}} \\
 \vdots\\
 \frac{b}{1-a} \\
 1 \\
\end{array}
\right),\; e_{2},\;\dots\;,\; e_{n-1}.$$ 
To discuss the irreducibility of the representation $\tilde{\xi}_1$, we write the matrices corresponding to generators of $T_n$ relative to the basis $B=\{w,e_2,\dots,e_{n-1}\}$ of $\mathbb{C}^{n-1}$.
Consider the transition matrix

$$P=(w,e_2,\cdots, e_j, \cdots , e_{n-1})=\left(
\begin{array}{cccccc}
 \frac{2 b^{n-2}}{(1-a)^{n-1}} & 0  &\cdots &0 &\cdots & 0\\
 \frac{b^{n-3}}{(1-a)^{n-3}} &1 &\cdots &0 &\cdots & 0\\
 \vdots & \vdots &\cdots & \vdots &\cdots & \vdots\\
 \frac{b^{n-j-1}}{(1-a)^{n-j-1}} & 0 &\cdots &1  &\cdots & 0\\
 \vdots & \vdots & \cdots & \vdots  &\cdots & \vdots\\
 \frac{b}{1-a} & 0 &\cdots &0  &\cdots & 0\\
 1 & 0 &\cdots &0  &\cdots & 1\\
\end{array}
\right)=
I_{n-1}+ve_1^T,$$
where $v=w-e_1$ and $I_{n-1}$ is the identity matrix of order $(n-1)$.
Now apply Sherman-Morrison formula \cite{Golub} to get the inverse of $P$.
\vspace{-.09cm}
$$P^{-1}=I_{n-1}-\frac{1}{1+e_1^Tv}ve_1^{T}=I_{n-1}-\frac{1}{w_1}ve_1^{T}=I_{n-1}-\frac{(1-a)^{n-1}}{2 b^{n-2}}ve_1^{T},$$
where $w_i$ stands for the $i$-th component of the vector $w$.
Hence, $P^{-1}$ is given by:
$$P^{-1}=\left(
\begin{array}{cccccc}
 \frac{(1-a)^{n-1}}{2 b^{n-2}} & 0  &\cdots &0 &\cdots & 0\\
 -\frac{(1-a)^{n-1}}{2b^{n-2}} &1 &\cdots &0 &\cdots & 0\\
 \vdots & \vdots &\cdots & \vdots &\cdots & \vdots\\
 -\frac{(1-a)^{j}}{2b^{j-1}} & 0 &\cdots &1  &\cdots & 0\\
 \vdots & \vdots & \cdots & \vdots  &\cdots & \vdots\\
-\frac{(1-a)^{2}}{2b} & 0 &\cdots &0  &\cdots & 1\\
\end{array}
\right).
$$
Therefore, the matrices corresponding to the generators $s_1,\dots,s_{n-1}$ relative to the basis $B$ are as follows.
$$s_1\mapsto S_1=P^{-1}\tilde{\xi}_1(s_1)P=
\left(
\begin{array}{cccc}
 -1 & 0 & \cdots & 0 \\
 0 & 1 & \cdots & 0  \\
 \vdots & \vdots & \ddots & \vdots  \\
 0 & 0 & \cdots & 1 \\
\end{array}
\right),$$ 
$$s_j\mapsto S_j=P^{-1}\tilde{\xi}_1(s_j)P \quad \text{ for } 2\leqslant j\leqslant n-1.$$ 

In the sequel, we provide the explicit form of the matrices $S_j$, ($2\leqslant j\leqslant n-1$). For this purpose, we use the matrix 
$M$ introduced in Theorem \ref{TheoTn}, namely
{
\setlength{\abovedisplayskip}{10pt}
\setlength{\belowdisplayskip}{10pt}
$$M=\begin{pmatrix}
a &b\\
\dfrac{1-a^2}{b} & -a
\end{pmatrix}.$$
}

The first column of the matrix  $S_2$ is given by:

{
\setlength{\abovedisplayskip}{10pt}
\setlength{\belowdisplayskip}{10pt}
\begin{align*}
S_2e_1= & P^{-1}\tilde{\xi}_1(s_2)Pe_1\\
= &(I_{n-1}-\frac{1}{w_1}ve_1^{T})(M\oplus I_{n-3})(I_{n-1}+ve_1^T)e_1\\
=& (I_{n-1}-\frac{1}{w_1}ve_1^{T})(M\oplus I_{n-3})(e_1+v)\\
=& (I_{n-1}-\frac{1}{w_1}ve_1^{T})(M\oplus I_{n-3})w\\
= & (I_{n-1}-\frac{1}{w_1}ve_1^{T})
\left(
\begin{array}{c}
aw_1+bw_2\\
\frac{1-a^2}{b}w_1-aw_2\\
w_3\\
\vdots\\
w_{n-1}
\end{array}
\right)\\
= & 
\left(
\begin{array}{c}
aw_1+bw_2\\
\frac{1-a^2}{b}w_1-aw_2\\
w_3\\
\vdots\\
w_{n-1}
\end{array}
\right)-
\frac{aw_1+bw_2}{w_1}(w-e_1)\\
= &
\left(
\begin{array}{c}
aw_1+bw_2\\
\frac{1-a^2}{b}w_1-aw_2\\
w_3\\
\vdots\\
w_{n-1}
\end{array}
\right)-
\frac{aw_1+bw_2}{w_1}\left(
\begin{array}{c}
w_1-1\\
w_2\\
w_3\\
\vdots\\
w_{n-1}
\end{array}
\right)\\
= &
\left(
\begin{array}{c}
\frac{1}{2} \left(a^2+1\right)\\
\frac{(3+a^2)(1+a)}{2(1-a)^3}\\
\frac{(1+a)b^{n-4}}{2(1-a)^{n-5}}\\
\vdots\\
\frac{(1+a)b^{n-j-1}}{2(1-a)^{n-j-2}}\\
\vdots\\
\frac{(1+a)}{2(1-a)^{-1}}
\end{array}
\right), \text{ where } 3\leqslant j\leqslant n-1.
\end{align*}
}

The second column of $S_2$ is given by:
{
\setlength{\abovedisplayskip}{10pt}
\setlength{\belowdisplayskip}{10pt}
\begin{align*}
S_2e_2= & P^{-1}\tilde{\xi}_1(s_2)Pe_2\\
= &(I_{n-1}-\frac{1}{w_1}ve_1^{T})(M\oplus I_{n-3})(I_{n-1}+ve_1^T)e_2\\
=& (I_{n-1}-\frac{1}{w_1}ve_1^{T})(M\oplus I_{n-3})(e_2+0)\\
=& (I_{n-1}-\frac{1}{w_1}ve_1^{T})(M\oplus I_{n-3})e_2
\end{align*}
\begin{align*}
=& (I_{n-1}-\frac{1}{w_1}ve_1^{T})
\left(
\begin{array}{c}
b\\
-a\\
0\\
\vdots\\
0
\end{array}
\right)\\
= & 
\left(
\begin{array}{c}
b\\
-a\\
0\\
\vdots\\
0
\end{array}
\right)-
\frac{b}{w_1}(w-e_1)\\
= &
\left(
\begin{array}{c}
b\\
-a\\
0\\
\vdots\\
0
\end{array}
\right)-
\frac{b}{w_1}\left(
\begin{array}{c}
w_1-1\\
w_2\\
w_3\\
\vdots\\
w_{n-1}
\end{array}
\right)\\
= &
\left(
\begin{array}{c}
\frac{b}{w_1}\\
-a-\frac{bw_2}{w_1}\\
-\frac{bw_3}{w_1}\\
\vdots\\
-\frac{bw_j}{w_1}\\
\vdots\\
-\frac{bw_{n-1}}{w_1}
\end{array}
\right)\\
= &
\left(
\begin{array}{c}
\frac{(1-a)^{n-1}}{2b^{n-3}}\\
\frac{-1-a^2}{2}\\
-\frac{(1-a)^3}{2b}\\
\vdots\\
-\frac{(1-a)^j}{2b^{j-2}}\\
\vdots\\
-\frac{(1-a)^{n-1}}{2b^{n-3}}
\end{array}
\right), \text{ where } 3\leqslant j\leqslant n-1.
\end{align*}
}

\noindent It remains to compute the $j$-th column of the matrix $S_2$, where $3\leqslant j\leqslant n-1$.

{
\setlength{\abovedisplayskip}{2pt}
\setlength{\belowdisplayskip}{10pt}
\begin{align*}
S_2e_j = &(I_{n-1}-\frac{1}{w_1}ve_1^{T})(M\oplus I_{n-3})(I_{n-1}+ve_1^T)e_j\\
=& (I_{n-1}-\frac{1}{w_1}ve_1^{T})(M\oplus I_{n-3})(e_j+0)\\
=& (I_{n-1}-\frac{1}{w_1}ve_1^{T})e_j\\
=& e_j.
\end{align*}}
Therefore the matrix $S_2$ is written, relative to the basis $B$, as follows.
{
\setlength{\abovedisplayskip}{15pt}
\setlength{\belowdisplayskip}{10pt}
$$S_2=\left(
\begin{array}{ccccccc}
\frac{1}{2} \left(a^2+1\right)& \frac{(1-a)^{n-1}}{2b^{n-3}} & 0 & \cdots & 0 & \cdots & 0\\
\frac{(3+a^2)(1+a)b^{n-3}}{2(1-a)^{n-2}} & \frac{-1-a^2}{2} & 0 & \cdots & 0 & \cdots & 0\\
\frac{(1+a)b^{n-4}}{2(1-a)^{n-5}} & -\frac{(1-a)^3}{2b} & 1 & \cdots & 0 & \cdots & 0\\
\vdots & \vdots & \vdots &\ddots & 0 & \cdots & \vdots \\
\frac{(1+a)b^{n-j-1}}{2(1-a)^{n-j-2}} & -\frac{(1-a)^j}{2b^{j-2}} & 0 & \cdots & 1 & \cdots &0\\
\vdots & \vdots & \vdots & \cdots & \vdots & \ddots & \vdots\\
\frac{(1+a)}{2(1-a)^{-1}} & -\frac{(1-a)^{n-1}}{2b^{n-3}} & 0 & \cdots & 0 & \cdots & 1
\end{array}
\right).
$$\\\\
Now, we compute the matrices $S_j$ relative to the basis $B$ for $3\leqslant j\leqslant n-1$.
$$ S_j=P^{-1}\tilde{\xi}_1(s_j)P=(I_{n-1}-\frac{1}{w_1}ve_1^{T})(I_{j-2}\oplus M\oplus I_{n-j-1})(I_{n-1}+ve_1^T).$$}

{\setlength{\abovedisplayskip}{2pt}
\setlength{\belowdisplayskip}{10pt}
The first column of the matrix $S_j$ is given as follows.

\begin{align*}
S_je_1 = & (I_{n-1}-\frac{1}{w_1}ve_1^{T})(I_{j-2}\oplus M\oplus I_{n-j-1})(I_{n-1}+ve_1^T)e_1\\
=& (I_{n-1}-\frac{1}{w_1}ve_1^{T})(I_{j-2}\oplus M\oplus I_{n-j-1})(e_1+v)\\
=& (I_{n-1}-\frac{1}{w_1}ve_1^{T})(I_{j-2}\oplus M\oplus I_{n-j-1})w\\
=& (I_{n-1}-\frac{1}{w_1}ve_1^{T})
\left(
\begin{array}{c}
w_1\\
\vdots\\
w_{j-2}\\
aw_{j-1}+bw_{j}\\
\frac{1-a^2}{b}w_{j-1}-aw_{j}\\
w_{j+1}\\
\vdots\\
w_{n-1}
\end{array}
\right)\\
\end{align*}}

{
\setlength{\abovedisplayskip}{10pt}
\setlength{\belowdisplayskip}{10pt}
\begin{align*}
= & 
\left(
\begin{array}{c}
w_1\\
\vdots\\
w_{j-2}\\
aw_{j-1}+bw_{j}\\
\frac{1-a^2}{b}w_{j-1}-aw_{j}\\
w_{j+1}\\
\vdots\\
w_{n-1}
\end{array}
\right)-\frac{1}{w_1}(w_1)v\\
= &
\left(
\begin{array}{c}
w_1\\
\vdots\\
w_{j-2}\\
aw_{j-1}+bw_{j}\\
\frac{1-a^2}{b}w_{j-1}-aw_{j}\\
w_{j+1}\\
\vdots\\
w_{n-1}
\end{array}
\right)
-
\left(
\begin{array}{c}
w_1-1\\
w_2\\
\vdots\\
w_{j}\\
\vdots\\
w_{n-1}
\end{array}
\right)\\
= &
\left(
\begin{array}{c}
1\\
0\\
\vdots\\
0\\
(a-1)w_{j-1}+bw_{j}\\
\frac{1-a^2}{b}w_{j-1}-(a+1)w_{j}\\
0\\
\vdots\\
0
\end{array}
\right)\\
= & 
\left(
\begin{array}{c}
1\\
0\\
\vdots\\
0
\end{array}
\right)\\
= & e_1.
\end{align*}}

Now we compute the $(j-1)$-th column of $S_j$.

{
\setlength{\abovedisplayskip}{10pt}
\setlength{\belowdisplayskip}{10pt}
\begin{align*}
S_je_{j-1}= &(I_{n-1}-\frac{1}{w_1}ve_1^{T})(I_{j-2}\oplus M\oplus I_{n-j-1})(I_{n-1}+ve_1^T)e_{j-1}\\
=& (I_{n-1}-\frac{1}{w_1}ve_1^{T})(I_{j-2}\oplus M\oplus I_{n-j-1})e_{j-1}
\end{align*}}
{
\setlength{\abovedisplayskip}{10pt}
\setlength{\belowdisplayskip}{10pt}
\begin{align*}
=& (I_{n-1}-\frac{1}{w_1}ve_1^{T})
\left(
\begin{array}{c}
0\\
\vdots\\
0\\
a \\
\frac{1-a^2}{b}\\
0\\
\vdots\\
0
\end{array}
\right)\\
= & 
\left(
\begin{array}{c}
0\\
\vdots\\
0\\
a \\
\frac{1-a^2}{b}\\
0\\
\vdots\\
0
\end{array}
\right)\\
= & ae_{j-1}+\frac{1-a^2}{b}e_j.
\end{align*} 
}

Similarly, we calculate the $j$-th column of $S_j$.

{
\setlength{\abovedisplayskip}{10pt}
\setlength{\belowdisplayskip}{10pt}
\begin{align*}
S_je_{j} = &(I_{n-1}-\frac{1}{w_1}ve_1^{T})(I_{j-2}\oplus M\oplus I_{n-j-1})(I_{n-1}+ve_1^T)e_{j}\\
=& (I_{n-1}-\frac{1}{w_1}ve_1^{T})(I_{j-2}\oplus M\oplus I_{n-j-1})e_{j}\\
=& (I_{n-1}-\frac{1}{w_1}ve_1^{T})
\left(
\begin{array}{c}
0\\
\vdots\\
0\\
b \\
-a\\
0\\
\vdots\\
0
\end{array}
\right)
\end{align*}}
{
\setlength{\abovedisplayskip}{10pt}
\setlength{\belowdisplayskip}{10pt}
\begin{align*}
= & 
\left(
\begin{array}{c}
0\\
\vdots\\
0\\
b \\
-a\\
0\\
\vdots\\
0
\end{array}
\right)\\
= & be_{j-1}-ae_j.
\end{align*}
}
The remaining columns of $S_j$ are given by $S_je_k$ for $2\leqslant k\leqslant n-1$ with $k\not= j, j-1$.

\begin{align*}
S_je_{k}= & P^{-1}\tilde{\xi}_1(s_j)Pe_{k}\\
= &(I_{n-1}-\frac{1}{w_1}ve_1^{T})(I_{j-2}\oplus M\oplus I_{n-j-1})(I_{n-1}+ve_1^T)e_k\\
=& (I_{n-1}-\frac{1}{w_1}ve_1^{T})(I_{j-2}\oplus M\oplus I_{n-j-1})e_k\\
=& (I_{n-1}-\frac{1}{w_1}ve_1^{T})e_k\\
= & e_k-0\\
= &  e_k.
\end{align*}
Therefore the matrix $S_j$ is written relative to the basis $B$ as
$$S_j=I_{j-2}\oplus M\oplus I_{n-j-1}=\tilde{\xi}_1(s_j)\;\;\text{ for } 3\leqslant j\leqslant n-1.$$

\vspace{0.1cm}

Throughout the next propositions, vectors and matrices are considered relative to the basis $B$ introduced earlier.

\begin{proposition}\label{prop e1}
Suppose that $a\not\in\{1,-1\}$, then any invariant subspace of the representation $\tilde{\xi}_1$  containing $e_1$  must also contain the following $n-3$ linearly independent vectors \vspace*{-0.2cm}
$$v_1=-be_2+(a+1)e_3,\; v_2=-be_3+(a+1)e_4,\;\dots,\;v_{n-3}=be_{n-2}+(a+1)e_{n-1}.$$
\end{proposition}

\begin{proof}
Let $V$ be an invariant subspace of $\tilde{\xi}_1$ such that $e_1\in V$. 
Then the vector $S_2e_1$ must be in $V$.
$$S_2e_1=\left(
\begin{array}{c}
 \frac{1}{2}(a^2+1) \\
 \frac{(3+a^2)(1+a)b^{n-3}}{2(1-a)^{n-2}} \\
 \frac{(1+a)b^{n-4}}{2(1-a)^{n-5}} \\
 \vdots\\
 \frac{(1+a)b^{n-j-1}}{2(1-a)^{n-j-2}}\\
 \vdots\\
   \frac{(1+a)}{2(1-a)^{-1}}\\
\end{array}
\right)
= \frac{1}{2}(a^2+1)e_1+\frac{(3+a^2)(1+a)b^{n-3}}{2(1-a)^{n-2}}e_2+
\sum_{j=3}^{n-1}\frac{(1+a)b^{n-j-1}}{2(1-a)^{n-j-2}}e_j.$$
Since $e_1\in V$ and $S_2e_1\in V$, it follows that the vector $f=S_2e_1-\frac{1}{2}(a^2+1)e_1\in V.$
The vector $f$ is simplified as follows.
 $$f=\frac{(3+a^2)(1+a)b^{n-3}}{2(1-a)^{n-2}}e_2+
\sum_{j=3}^{n-1}\frac{(1+a)b^{n-j-1}}{2(1-a)^{n-j-2}}e_j.
$$
Again, since $f\in V$ and $V$ is invariant subspace, it follows that the  vector $S_3f-f$ must be in $V.$
\begin{align*}
S_3f & =S_3\left(
\frac{(3+a^2)(1+a)b^{n-3}}{2(1-a)^{n-2}}e_2+
\sum_{j=3}^{n-1}\frac{(1+a)b^{n-j-1}}{2(1-a)^{n-j-2}}e_j
\right)\\\\
& =(I_{1}\oplus M\oplus I_{n-4})\left(
\frac{(3+a^2)(1+a)b^{n-3}}{2(1-a)^{n-2}}e_2+
\sum_{j=3}^{n-1}\frac{(1+a)b^{n-j-1}}{2(1-a)^{n-j-2}}e_j
\right)\\\\
 & = \left(M\left(
\begin{array}{c}
\frac{(3+a^2)(1+a)b^{n-3}}{2(1-a)^{n-2}}\\\\
\frac{(1+a)b^{n-4}}{2(1-a)^{n-5}}
\end{array}
\right)\right)_1e_2
+\left(M\left(
\begin{array}{c}
\frac{(3+a^2)(1+a)b^{n-3}}{2(1-a)^{n-2}}\\\\
\frac{(1+a)b^{n-4}}{2(1-a)^{n-5}}
\end{array}
\right)\right)_2e_3
+\sum_{j=4}^{n-1}\frac{(1+a)b^{n-j-1}}{2(1-a)^{n-j-2}}e_j\\\\
 & =
\left(\left(\begin{array}{c}
a\frac{(3+a^2)(1+a)b^{n-3}}{2(1-a)^{n-2}}+b\frac{(1+a)b^{n-4}}{2(1-a)^{n-5}}\\
\frac{1-a^2}{b}\frac{(3+a^2)(1+a)b^{n-3}}{2(1-a)^{n-2}}-a(\frac{(1+a)b^{n-4}}{2(1-a)^{n-5}})
\end{array}
\right)\right)_1e_2\\\\
&\:\:\:\:
+\left(\left(\begin{array}{c}
a\frac{(3+a^2)(1+a)b^{n-3}}{2(1-a)^{n-2}}+b\frac{(1+a)b^{n-4}}{2(1-a)^{n-5}}\\
\frac{1-a^2}{b}\frac{(3+a^2)(1+a)b^{n-3}}{2(1-a)^{n-2}}-a(\frac{(1+a)b^{n-4}}{2(1-a)^{n-5}})
\end{array}
\right)\right)_2e_3
+\sum_{j=4}^{n-1}\frac{(1+a)b^{n-j-1}}{2(1-a)^{n-j-2}}e_j\\\\
& = \frac{(1+a)(1+3a^2)b^{n-3}}{2(1-a)^{n-2}}e_2 +\frac{(1+a)(3a^2+2a+3)b^{n-4}}{2(1-a)^{n-3}}e_3 +\sum_{j=4}^{n-1}\frac{(1+a)b^{n-j-1}}{2(1-a)^{n-j-2}}e_j.
\end{align*}
\vspace*{-0.15cm}
Therefore,
\vspace*{-0.15cm}
$$S_3f-f=\left(
\begin{array}{c}
 0 \\
 -\frac{(1+a)^2 b^{n-3}}{(1-a)^{n-3}} \\
 \frac{(1+a)^3 b^{n-4}}{(1-a)^{n-3}} \\
 0 \\
 \vdots \\
 0 \\
\end{array}
\right)\in V.\vspace*{-0.15cm}$$
Hence, the vector $v_1=\frac{(1-a)^{n-3}}{b^{n-4}(1+a)^2}(S_3f-f)=-be_2+(1+a)e_3\in V$. Then, direct computations lead to the following.\vspace*{-0.1cm}
$$S_4v_1-v_1=(I_{2}\oplus M\oplus I_{n-5}-I_{n-1})
\left(\begin{array}{c}
0\\
-b\\
1+a\\
0\\
\vdots\\
0
\end{array}
\right)=
\left(\begin{array}{c}
0\\
0\\
(a-1)(1+a)\\
\frac{(1-a^2)(1+a)}{b}\\
0\\
\vdots\\
0
\end{array}
\right)\in V.\vspace*{-0.1cm}
$$
Thus, $v_2=\frac{b}{(1-a)(1+a)}(S_4v_1-v_1)=-be_3+(1+a)e_4\in V$.
Now assume that $v_k=-be_{k+1}+(1+a)e_{k+2}\in V$ for $1\leqslant k<n-3.$ 
Direct computations lead to the following.
$$S_{k+3}v_k-v_k=(I_{k+1}\oplus M\oplus I_{n-k-4}-I_{n-1})
\left(\begin{array}{c}
0\\
\vdots\\
0\\
-b\\
1+a\\
0\\
\vdots\\
0
\end{array}
\right)=
\left(\begin{array}{c}
0\\
\vdots\\
0\\
(a-1)(1+a)\\
\frac{(1-a^2)(1+a)}{b}\\
0\\
\vdots\\
0
\end{array}
\right)\in V.
$$
Hence, $v_{k+1}=\frac{b}{(1-a)(1+a)}(S_{k+3}v_k-v_k)=-be_{k+2}+(1+a)e_{k+3}\in V$.
Therefore, by finite mathematical induction, we conclude that 
$v_k=-be_{k+1}+(1+a)e_{k+1}\in V$ for all $1\leqslant k\leqslant n-3$.
It is clear that these $(n-3) $ vectors are linearly independent.
\end{proof}

\vspace*{0.1cm}

We now introduce a lemma that we need in the rest of the paper.

\begin{lemma}\label{det}
Let $M_n=(m_{ij})_{i,j=1}^{n}$ be an $n\times n$ matrix with 
$$m_{i1}=x_i\ (1\leqslant i\leqslant n), \ m_{12}=1, \text{ and for  }\  2\leqslant i\leqslant n, \;\; 3\leqslant j\leqslant n,$$
$$ m_{ij}=\left\{\begin{array}{cr}
y_2 & j=i\\
y_1 & j=i+1\\
0 & \text{ otherwise}
\end{array}\right..
$$
That is 
$$ M_n=\begin{pmatrix}
x_1 & 1 & 0 & 0 & \cdots & 0\\
x_2 & 0 & y_1 & 0 & \cdots & 0\\
x_3 & 0 & y_2 & y_1 &\cdots & 0\\
x_4 & 0 & 0 & y_2 & \ddots& \vdots\\
\vdots & \vdots & \ddots & \ddots &\ddots & y_1\\
x_n & 0 & \cdots & 0 &0& y_2
\end{pmatrix}.$$ 
Then the determinant of $M_n$ is given by
\[
\boxed{\;\det M_n=\sum_{k=2}^{n}(-1)^{k+1}x_k\,y_1^{\,k-2}y_2^{\,n-k}\;.}
\]
\end{lemma}

\begin{proof}
Denote the minors of $M$ by $\det M_n^{(k)}$, $1\leqslant k\leqslant n$. The determinant of $M_n$ is 
\[\det M_n=\sum_{k=1}^n(-1)^{k+1}x_k\det M_n^{(k)}.
\]
For \(k=1\), the minor \(\det M_n^{(1)}=0\) (column~2 becomes \(e_1\) in the \((n-1)\times(n-1)\) minor).  
For \(k\ge2\), expand \(\det M_n^{(k)}\) along its first column (which has a single 1 in the top entry); the remaining \((n-2)\times(n-2)\) matrix is lower-bidiagonal with \(k-2\) factors \(y_1\) and \(n-k\) factors \(y_2\), hence \(\det M_n^{(k)}=y_1^{\,k-2}y_2^{\,n-k}\).  
Therefore
\[
\det M_n=\sum_{k=2}^{n}(-1)^{k+1}x_k\,y_1^{\,k-2}y_2^{\,n-k}.
\]
\end{proof}

\begin{proposition}\label{prop2 e1}
Suppose that $a\not\in\{1,-1\}$. The invariant subspaces of  $\tilde{\xi}_1$ that contain $e_1$ are proper subspaces if and only if $a$ is a root of the polynomial
$$P(t)=4(1+t^2) + \frac{(1-t)^4}{2t} \left( 1 - \left( \frac{1-t}{1+t} \right)^{n-4} \right).$$
\end{proposition}

\begin{proof}
Let $V$ be an invariant subspace such that $e_1\in V$. Then, by Proposition \ref{prop e1}, the vectors $v_k=-be_{k+1}+(1+a)e_{k+2}$, $1\leqslant k\leqslant n-3$, are $n-3$ linearly independent vectors in $V$.
Since $V$ is invariant subspace, it follows that $S_ke_1\in V$ and $S_kv_j\in V$ for all $1\leqslant k\leqslant n-1$ and $1\leqslant j\leqslant n-3$. 
Consider the subspace $W=\langle e_1, v_1,\dots,v_{n-3}\rangle$ of $\mathbb{C}^{n-1}$.
Then either $V=W$ or $V=\mathbb{C}^{n-1}$.
It is clear that $S_1e_1=-e_1\in W$ and $S_1v_j=v_j\in W$ for all $1\leqslant j\leqslant n-3$.
In addition, direct computations yields the following results.
\begin{itemize}[noitemsep, topsep=5pt]
\item $S_2v_j=v_j\in W$ for all $2\leqslant j\leqslant n-3.$
\item $S_kv_{k-3}=v_{k-3}+\frac{1-a^2}{b}v_{k-2}\in W$, for all $4\leqslant k\leqslant n-1$.
\item $S_kv_{k-2}=-v_{k-2}\in W$, for all $3\leqslant k\leqslant n-1$.
\item $S_kv_{k-1}=bv_{k-2}+v_{k-1}\in W$, for all $3\leqslant k\leqslant n-1$.
\item $S_kv_j=v_j\in W$ for all $3\leqslant k\leqslant n-1$ and all $1\leqslant j\leqslant n-1$ with $j\not\in\{k-1,k-2,k-3\}$.
\end{itemize} 
It remains to check the case $S_2v_1$. We have
$$S_2v_1=
\left(
\begin{array}{c}
-\frac{(1-a)^{n-1}}{2b^{n-4}}\\
\frac{(1+a^2)b}{2}\\
\frac{(1-a)^3}{2}+a+1\\
\frac{(1-a)^4}{2b}\\
\vdots\\
\frac{(1-a)^j}{2b^{j-3}}\\
\vdots\\
\frac{(1-a)^{n-1}}{2b^{n-4}}
\end{array}
\right).
$$
Then, $S_2v_1$ belongs to $W$ if and only if $S_2v_1,e_1,v_1,\cdots,v_{n-3}$ are linearly dependent vectors in $\mathbb{C}^{n-1}$. This is equivalent to the following equation.
$$\det(S_2v_1,e_1,v_1,\cdots,v_{n-3})=0.$$
By Lemma \ref{det}, the determinant $\Delta = \det(S_2v_1,e_1,v_1,\cdots,v_{n-3})$ is given by
$$\Delta=\sum_{k=2}^{n-1}(-1)^{k+1}x_k\,(-b)^{k-2}(1+a)^{n-1-k},$$ where $x_k$ is the $k$-th component of $S_2v_1$.
The determinant above can be simplified as follows.

{
\setlength{\abovedisplayskip}{5pt}
\setlength{\belowdisplayskip}{5pt}
\begin{align*}
\Delta&=\sum_{k=2}^{n-1}(-1)^{k+1}x_k\,(-b)^{k-2}(1+a)^{n-1-k}\\
&=\frac{-b}{2}\left[
(1+a^2)(1+a)^{n-3} +((1-a)^3+2a+2)(1+a)^{n-4}\right.\\
&\quad \left.+(1-a)^4(1+a)^{n-5}+\cdots
+(1-a)^k(1+a)^{n-1-k}+\cdots+(1-a)^{n-1}
\right] \\
& =\frac{-b}{2}\left[
(1+a^2)(1+a)^{n-3}+
(3-a)(1+a^2)(1+a)^{n-4}
+\sum_{k=4}^{n-1}(1-a)^k(1+a)^{n-1-k}\right]\\
& = \frac{-b}{2}\left[4(1+a^2)(1+a)^{n-4}+\sum_{k=4}^{n-1}(1-a)^k(1+a)^{n-1-k}\right]\\
& = \frac{-b}{2}(1+a)^{n-4}\left[4(1+a^2)+(1+a)^3\sum_{k=4}^{n-1}\left(\frac{1-a}{1+a}\right)^{k}\right].
\end{align*}}
Now, set $m=k-3$ and $r=\frac{1-a}{1+a}$. Then the geometric sum appeared in the last equation is simplified as follows.

$$\sum_{k=4}^{n-1}\left(\frac{1-a}{1+a}\right)^{k}=\sum_{m=1}^{n-4}r^{m+3} = r^3\sum_{m=1}^{n-4}r^{m} = \left\{
\begin{array}{ll}
r^4\dfrac{1-r^{n-4}}{1-r}&\text{ if } r\not=1\\
n-4 & \text{ if } r=1
\end{array}\right..$$
So,
$$(1+a)^3\sum_{k=4}^{n-1}\left(\frac{1-a}{1+a}\right)^{k} =\left\{
\begin{array}{ll}
\dfrac{(1-a)^{4}}{2a}\left(1-\left(\dfrac{1-a}{1+a}\right)^{n-4}\right) & \text{ if } a\neq0\\
n-4& \text{ if } a=0
\end{array} 
\right..
$$
Hence, 
$$\Delta = \left\{
\begin{array}{ll}
\dfrac{-b}{2}(1+a)^{n-4}\left[4(1+a^2)+\dfrac{(1-a)^{4}}{2a}\left(1-\left(\dfrac{1-a}{1+a}\right)^{n-4}\right)\right] & \text{ if } a\neq0\\
\dfrac{-b}{2}n & \text{ if } a=0
\end{array}\right..$$
Therefore, $S_2v_1\in W$ if and only if $a\neq0$ and 
$$ P(a)=4(1+a^2) + \frac{(1-a)^4}{2a} \left( 1 - \left( \frac{1-a}{1+a} \right)^{n-4} \right)=0.$$
Consequently, $V=W\neq \mathbb{C}^{n-1}$ if and only if $a\neq0$ and $P(a)=0$. 

\end{proof}

\begin{proposition}\label{prop no e1}
Suppose that $a\not=1$. Then any invariant subspace of the representation $\tilde{\xi}_1$ that doesn't contain the vector $e_1=(1,0,\dots,0)^T$ must be the trivial subspace.
\end{proposition}
\begin{proof}
Consider an invariant subspace $V$ that does not contain the vector $e_1$. Since $e_1$ is an eigenvector of $\tilde{\xi}_1(s_1)$, it follows that $V$ is spanned by linear combinations of the other eigenvectors of $\tilde{\xi}_1(s_1)$ which are $e_2,\dots, e_{n-1}$. This implies that any vector $v$ in $V$ has the form
$$v=x_2e_2+\cdots +x_{n-1}e_{n-1}.$$
$S_kv\in V$ for all $1\leqslant k\leqslant n-1$ and for all $v\in V.$
The first component of the vector $S_2v$ is $\frac{(1-a)^{n-1}}{2b^{n-3}}x_2$. So, $x_2=0$ because $a\neq 1$. Thus all vectors $v$ in $V$ have the form 
$v=x_3e_3+\cdots +x_{n-1}e_{n-1}$.
Assume, for some $j\in \{2,\dots, n-2\}$, that the $s$-th component of any vector $v$ in $V$ is zero for all $1\leqslant s\leqslant j$.
Then the $j$-th component of $S_{j+1}v$ is $bx_{j+1}$. Thus, by our assumption, $x_{j+1}=0$ because $S_{j+1}v\in V$ and $b\neq0$. Therefore, by finite mathematical induction, all the components $x_k, k\in\{1,2,\dots,n-1\}$ of any vector $v$ in $V$ are zeros. Therefore, $V$ is the trivial subspace.
\end{proof}
Now, we state the main theorem of our work.
\begin{theorem}
The representation $\tilde{\xi}_1$ of the twin groups $T_n$ ($n\geqslant4$) is irreducible if and only if $a\not\in\{1,-1\}$
and $a$ is not a root of:
$$P(t)=4(1+t^2) + \frac{(1-t)^4}{2t} \left( 1 - \left( \frac{1-t}{1+t} \right)^{n-4} \right).$$
\end{theorem}
\begin{proof}
The proof follows directly from Proposition \ref{prop -1,1}, Proposition \ref{prop e1}, Proposition \ref{prop2 e1}, and Proposition \ref{prop no e1}.
\end{proof}


\end{document}